\documentclass[10pt]{article}
\usepackage{amsthm,amsfonts,amsmath,amscd,amssymb,epsfig,psfrag}
\usepackage{enumerate}

\usepackage{hyperref}
\usepackage[all]{xy}
\theoremstyle{plain}
\newtheorem{thm}{Theorem}

\newtheorem{lem}[thm]{Lemma}

\newtheorem{ex}[thm]{Example}
\newtheorem{defn}[thm]{Definition}

\theoremstyle{remark}

\newtheorem*{rem}{Remark}
\newtheorem*{rems}{Remarks}

\theoremstyle{definition}

\newcommand{\id}{{{\mathchoice {\rm 1\mskip-4mu l} {\rm 1\mskip-4mu l}
{\rm 1\mskip-4.5mu l} {\rm 1\mskip-5mu l}}}}

\newcommand{\wh}{\widehat}
\newcommand{\wt}{\widetilde}
\newcommand{\ol}{\overline}
\newcommand{\ul}{\underline}

\newcommand{\om}{\omega}

\newcommand{\eps}{\varepsilon}

\newcommand{\lra}{\longrightarrow}
\newcommand{\N}{{\mathbb{N}}}
\newcommand{\Z}{{\mathbb{Z}}}
\newcommand{\R}{{\mathbb{R}}}

\newcommand{\Q}{{\mathbb{Q}}}

\renewcommand{\a}{{\bf a}}

\newcommand{\im}{{\rm im\, }}        

\newcommand{\LL}{\mathcal{L}}

\renewcommand{\AA}{\mathcal{A}}

\newcommand{\PP}{\mathcal{P}}

\newcommand{\ab}{{\mathbf a}}
\newcommand{\ib}{{\mathbf i}}

\newcommand{\hb}{{\mathbf h}}
\newcommand{\fb}{{\mathbf f}}
\newcommand{\gb}{{\mathbf g}}

\newcommand{\vect}{\overrightarrow}
\newcommand{\tcev}{\overleftarrow}

\newcommand{\del}{\partial}

\newcommand{\lin}{\mathrm{lin}}

\newcommand{\comp}{\mathrm{comp}}
\newcommand{\comment}[1]{}
\begin{document}

\title{Remarks on the rational SFT formalism}
\author{Janko Latschev\footnote{Fachbereich Mathematik, Universit\"at Hamburg, Bundesstra{\ss}e 55, 20146 Hamburg, Germany}
}
\maketitle
\begin{abstract}
Recently, rational symplectic field theory (SFT) has played a central role in the work of Siegel~\cite{Sie19} and of Moreno and Zhou~\cite{MZ20}, who defined (different) new invariants of contact manifolds using the theory. The goal of this note is to show how these are related to the original formulation of the theory by Eliashberg, Givental and Hofer~\cite{EGH00}.
\end{abstract}

\section{Introduction}

Symplectic field theory (SFT) as described by Eliashberg, Givental and Hofer~\cite{EGH00} gives a general framework for discussing holomorphic curve invariants of contact and symplectic manifolds. The theory comes in several levels of generality, ranging from full SFT, which takes into account curves with underlying domains of arbitrary genus, to (linearized) contact homology, which deals with a specific subcollection of curves defined on punctured spheres. The analytic foundations of the theory have still not been published completely, which has rather slowed down the development of the subject.

Recently, however, there have been some works that use the intermediate world of {\em rational} SFT, i.e.~the part of the theory that deals with all curves of genus zero. In \cite{Sie19}, Siegel used rational SFT to define a large number of new symplectic capacities, and showed that they give interesting new embedding obstructions. In \cite{MZ20}, Moreno and Zhou describe and use new numerical invariants of contact manifolds which obstruct the existence of fillings or cobordisms, and which are also extracted from rational SFT. Both of these papers use a rather ad hoc description of the theory developed by Siegel, who implemented suggestions of Hutchings~\cite{Hu13}.

The goal of this note is to show how the relevant parts of the theory can be formulated in the algebraic language of the original description of rational SFT by Eliashberg, Givental and Hofer, and in this way establish a clearer link between the two descriptions of the theory. At the same time, I believe that the algebraic formalism used here makes the symmetries and combinatorial factors in the various constructions a bit more transparent than the graphical approach chosen in \cite{Sie19}. 

In the following section, which forms the core of this work, I describe the relevant aspects of the algebraic formalism. After a brief review of rational SFT in \S~\ref{sec:rsft}, I then explain how the algebraic aspects of the work of Siegel (in \S~\ref{sec:siegel}) and of Moreno-Zhou (in \S~\ref{sec:moreno_zhou}) are dealt with in this formalism. I refrain from discussing any of the geometric arguments given or consequences derived in these works, as at the moment I have nothing to add in this regard.

{\bf Acknowledgments:} It will be obvious from what follows that the immediate inspiration for this note are the two papers by Siegel~\cite{Sie19} and Moreno-Zhou~\cite{MZ20}. My view of the subject was essentially formed in many discussions with Kai Cieliebak over the years, and this note would not exist without our collaboration. I also thank Kyler Siegel for helpful discussions on his work.

\section{Algebra}
\label{sec:algebra}

Throughout this section, we fix a unital ring $\Lambda$ containing $\Q$, and an integer $N\in \Z$. For a $\Z$-graded $\Lambda$-module $V=\oplus_{d\in\Z} V_d$, our convention is that the degree shifted module $V[n]$ has graded components $V[n]_d=V_{d+n}$, so the degree of individual elements is lower by $n$ after the shift.

\subsection{Structures}
\label{ssec:structures}

Let $C$ be a free graded module over the ring $\Lambda$, generated by symbols
$q_\gamma$ with $\gamma$ in some indexing set $\Gamma$. Set $\AA:=SC$,
which we will view as either an algebra or a coalgebra depending on the
context. A linear basis consists of all monomials in the $q_\gamma$.
Moreover, let $\PP:= SC \otimes S(C[-2N])^*$ be (a degree shifted version of)
the linear self-maps of $\AA=SC$.
We identify these with power series in variables $p_\gamma$ with polynomial coefficients in the $q_\gamma$. The degree shift means that
$$
|p_\gamma|=2N-|q_\gamma|.
$$
Thinking of $\PP$ as an algebra over $\Lambda$, we can endow it with a Poisson bracket of degree $-2N$, characterized by the properties that all $p_\gamma$ (respectively all $q_\gamma$) pairwise commute, and
$$
\{p_{\gamma_1},q_{\gamma_2}\} = \kappa_{\gamma_1} \delta_{\gamma_1\gamma_2},
$$
where $\kappa:\Gamma \to \Z$ is a choice of weights.
\begin{rem}
I have included the weights $\kappa_\gamma$ for easier comparison with the formalism of SFT as described in \cite{EGH00}. One could avoid them by absorbing them into either the $p$- or the $q$-variables, for example redefining $p_\gamma$ as $\frac 1{\kappa_\gamma} p_\gamma$.
\end{rem}

Any element $h \in \PP$ 
can be expanded as
$$
h= \sum_{r,s \ge 0} h^r_s,
$$
where $h^r_s$ is of degree $r$ in $p$ and of degree $s$ in $q$.
%
Moreover, we have the ideals
$$
\ol{\PP}:= \{ h \in \PP\,:\, h^0_s=0 \text{ for all } s\ge 0\}
$$
and
$$
\ul{\PP}:= \{ h \in \PP\,:\, h^r_0=0 \text{ for all } r\ge 0\},
$$
Thinking in terms of power series,
$$
h\in \ol{\PP} \quad \iff \quad h|_{p=0}=0 \qquad \text{and} \qquad
h\in \ul{\PP} \quad \iff \quad h|_{q=0}=0.
$$ 
Elements of the intersection
$$
\wh{\PP}:= \ol{\PP} \cap \ul{\PP}
$$
correspond to multiplinear maps from $C$ to $C$ which are symmetric in inputs and outputs.

We now consider an element $h \in \ol{\PP}$ which is homogeneous of degree $d$. For any $r\ge 1$, we  use $h^r=\sum_{s\ge 0} h^r_s$ to define a map
$$
\vect{h^r}: S^r\AA \to \AA
$$
of degree $d-2Nr$ by using iterated brackets, i.e.
$$
\vect{h^r}(w_1 \odot \dots \odot w_r) = \{\dots\{ h^r,w_1\}, \dots\},w_r\},
$$
where $w_i\in \AA$. The fact that this operation is indeed symmetric in the $w_i$ is a direct consequence of the Jacobi identity for the Poisson bracket and the fact that the $w_i$ pairwise Poisson commute. It is convenient to shift degrees and consider the map of degree $d-2N$
$$
\vect{h}=\sum_{r\ge1} \vect{h^r}:\ol{S}(\AA[2N]) \to \AA[2N],
$$
where we use the notation $\ol{S}V = \oplus_{r\ge 1}S^rV$. Viewing $\ol S(\AA[2N])$ as a coalgebra, this map admits a unique extension to a coderivation
$$
D_h=\sum_{r=1}^\infty D^r_h:\ol{S}(\AA[2N]) \to \ol{S}(\AA[2N])
$$
of the same degree, where for $n\ge r$ the map $D^r_h:S^n(\AA[2N]) \to S^{n-r+1}(\AA[2N])$ is given by
\begin{equation}\label{eq:coder}
D^r_h(w_1 \odot \dots \odot w_n) = \hspace{-5mm}\sum_{{i_1<i_2<\dots<i_r}\atop {i_{r+1}<i_{r+2}<\dots<i_N}}  \hspace{-5mm}(-1)^{\boxtimes}\,\,\vect{h^r}(w_{i_1} \odot \dots \odot w_{i_r}) \odot w_{i_{r+1}} \odot \dots \odot w_{i_n}.
\end{equation}
Here the sum is over all $(r,n-r)$-unshuffles, and $(-1)^\boxtimes$ is the sign acquired from permuting the words $w_i$ into the given order (the even shift in degrees is harmless here). By a routine computation, we get
\begin{lem}\label{lem:codiff}
For homogeneous elements  $h$ and $g$ of $\ol{\PP}$ we have
$$
D_h \circ D_g- (-1)^{|h|\cdot |g|}D_g \circ D_h=D_{\{h,g\}}.
$$
In particular, if $h\in \ol\PP$ has odd degree and $\{h,h\}=0$, then $D_h$ is a codifferential. \phantom{blah} \hfill$\qed$
\end{lem}
\begin{rem}
By our assumption $h\in \ol\PP$, the codifferential associated to an odd Poisson-self-commuting element $h$ has the additional property $D_h(\id)=0$, where $\id\in \Lambda[2N] \subseteq \AA[2N]\cong S^1(\AA[2N])$ is the unit of the algebra $\AA$.
\end{rem}
Suppose $\hb\in \ol{\PP}$ is of degree $2N-1$ and satisfies $\{\hb,\hb\}=0$, so that $D=D_\hb$ has degree $-1$. It is well-known (cf. \cite{MSS02}) that this means that the operations $\{\vect{\hb^r}\}_{r\ge 1}$ can be interpreted as an $L_\infty$-structure on $\AA[2N+1]$. This is the main structure that we will study and refine in the following subsections. Note that, by construction, the individual operations $\vect{\hb^r}$ are derivations with respect to the algebra structure on $\AA$ in each entry.

\begin{rem}
It is worth pointing out that, at this point, we have the following invariants associated to an element $\hb\in \ol\PP$ of degree $2N-1$ satisfying $\{\hb,\hb\}=0$:
\begin{enumerate}[(i)]
\item Taking a single Poisson bracket with $\hb$ defines a derivation $\mathbf D:\PP \to \PP$ of degree $-1$, and so there is an associated homology $H_*(\PP,\mathbf D)$. This is the point of view taken in the description of rational SFT in \cite{EGH00}.
\item We also have the homology $H_*(\ol S(\AA[2N]),D_\hb)$ of the symmetric coalgebra $\ol S(\AA[2N])$ with respect to the coderivation introduced above. Note that, by construction,  $D_\hb$ preserves each of the truncations $S^{\le k}(\AA[2N])$, so we also have the corresponding homologies. 
\item The operator $D^1=\vect{\hb^1}:\AA \to \AA$ induced from the part $\hb^1$ of $\hb$ linear in the $p$-variables is a differential of degree $-1$, and so one can consider its homology $H_*(\AA,D^1)$. In the context of SFT, this is what is known as (full) contact homology.
\end{enumerate}
If desired, one could use homotopy transfer to induce an $L_\infty$-structure on $H_*(\AA[2N+1],D^1)$ from the one on $\AA[2N+1]$. 
\end{rem}

\subsection{Morphisms}
\label{ssec:mor}

Now suppose we have two graded $\Lambda$-modules $C^+$ and $C^-$, together with the associated $\AA^\pm$ and $\PP^\pm$. We introduce the module $\LL:= SC^- \otimes S(C^+[-2N])^*$ of power series in the $p^+$ with polynomial coefficients in the $q^-$. As with $\PP^\pm$ above, we introduce the notation
$$
\ol{\LL}:=\{f\in \LL \,:\, f|_{p^+=0}=0\}, \quad
\ul{\LL}:=\{f\in \LL \,:\, f|_{q^-=0}=0\} \quad\text{and}\quad
\wh{\LL}:= \ol{\LL} \cap \ul{\LL}.
$$

To an element $g=g_s\in \PP^+$ which is of pure degree $s$ in the $q$-variables, we can associate the map
$$
\LL \gets S^s \LL: \tcev{g_s}
$$
given by iterated brackets
$$
(c_1 \odot \dots \odot c_s)\tcev{g_s} = \{c_1, \{\dots,\{c_s,g_s\}\dots\}.
$$
\begin{rem}
Note that $s=0$ is allowed here, and gives the inclusion of a power series $g_0$ in $p^+$ independent of $q^+$ into $S^1\LL \cong \LL$.
\end{rem}
For $s\ge 1$, these maps are derivations in each entry. For all $s\ge 0$, they can be extended to coderivations
$$
S\LL \gets S\LL : \tcev{D_{g_s}}
$$
in the usual way (cf.~\eqref{eq:coder}).

Similarly, given an element $g=g^r \in \PP^-$ which is of pure degree $r$ in $p^-$, we associate to it the map
$$
\vect{g^r}:S^r\LL \to \LL
$$
given by iterated brackets
$$
\vect{g^r}(c_1 \odot \dots \odot c_r) = \{\dots\{g^r,c_1\},\dots\},c_r\}.
$$
These can also be extended to coderivations
$$
\vect{D_{g^r}}: S\LL \to S\LL.
$$
The analogue of Lemma~\ref{lem:codiff} in this context is
\begin{lem}\label{lem:symmetric}
If $g_1$ and $g_2$ are two homogeneous elements of $\PP^+$ with $\{g_1,g_2\}=0$, then
$$
\tcev{D_{g_1}} \circ \tcev{D_{g_2}} = (-1)^{|g_1||g_2|} \tcev{D_{g_2}} \circ \tcev{D_{g_1}}.
$$
The corresponding result also holds for pairs of elements of $\PP^-$. \hfill $\qed$
\end{lem}
As a consequence of this observation, the action of $\PP^+$ on $S\LL$ defined above gives rise to an action of $\ol{S}\AA^+$ on $S\LL$, given by
$$
(c_1 \odot \dots \odot c_n) \tcev{D}_{w_1 \odot \dots \odot w_k} := (c_1 \odot \dots \odot c_n) \tcev{D_{w_1}} \circ \dots \circ \tcev{D_{w_k}}.
$$
\begin{rem}
The total number of brackets taken for $\tcev{D}_{w_1 \odot \dots \odot w_k}$ corresponds to the total word length in the $q^+$-variables of the words $w_i$.
\end{rem}

Given an element $\fb\in \LL$ of degree $2N$, we now consider the expression
$$
e^\fb = 1 + \sum_{k=1}^\infty \frac 1{k!} \underbrace{\fb\odot \dots \odot \fb}_{k \text{ times}},
$$
where $1\in \Lambda \cong S^0\LL$ is the unit of our ground ring.
\begin{rem}
As an infinite sum, this is technically not an element of $S\LL$, but rather
lives in a completion of $S\LL$ with respect to the word length filtration.
The coderivations $\vect{D_{g^r}}$ and $\tcev{D_{g_s}}$ introduced above
still make sense on this completion. As we also use other completions later on, we choose not to overburden the notation and will write things like $e^\fb\in S\LL$ without explicitly mentioning completions.
\end{rem}
One easily verifies that for $g^r\in \PP^-$ of degree $r$ in $p^-$ an $g_s\in \PP^+$ of degree $s$ in $q^+$ we have
$$
\vect{D_{g^r}}(e^\fb) = \vect{g^r}(\frac 1{r!}\fb^{\odot r}) \odot e^\fb \qquad \text{and} \qquad
(e^\fb)\tcev{D_{g_s}} = e^\fb \odot (\frac 1{s!}\fb^{\odot s})\tcev{g_s}.
$$
\begin{rem}
As we assume that $\fb$ has even degree, the position of the $e^\fb$-factor is rather cosmetic.

Note that we abuse notation slightly here, as both $\vect{D_{g^r}}(e^\fb)$ and $(e^\fb)\tcev{D_{g_s}}$ are elements of $\ol S\LL \subset S\LL$, with leading terms $\vect{g^r}(\frac 1{r!}\fb^{\odot r})$ and $(\frac 1{s!}\fb^{\odot s})\tcev{g_s}$ in $S^1\LL =\LL$, respectively.
\end{rem}

\begin{rem}
The terms $\vect{g^r}(\frac 1{r!}\fb^{\odot r})$ and $(\frac 1{s!}\fb^{\odot s})\tcev{g_s}$ have the following alternative description, which helps to establish the relationship to the SFT formalism as formulated in \cite{EGH00}.
Consider a summand $w \cdot w'$ in $g^r$, where $w$ is a word in $q^-$ and $w'=p_{\gamma_1}^-\cdots p_{\gamma_r}^-$ is a word of length $r$ in $p^-$. Then
\begin{align*}
\vect{ww'}(\frac 1{r!}\fb^{\odot r})
&= \frac 1{r!} w \cdot \{\dots\{w',\fb\},\dots\},\fb\} \\
&= \frac 1{r!}w \cdot \sum_{\sigma\in S_r} \kappa_{\gamma_1}\cdots\kappa_{\gamma_r}\frac {\del \fb}{\del p_{\gamma_1}}\cdots \frac{\del \fb}{\del p_{\gamma_r}}\\
&= w \cdot \kappa_{\gamma_1}\cdots\kappa_{\gamma_r}\frac {\del \fb}{\del p_{\gamma_1}}\cdots \frac{\del \fb}{\del p_{\gamma_r}}.
\end{align*}
So
$$
\vect{g^r}(\frac 1{r!}\fb^{\odot r}) = g^r|_{L_\fb},
$$
where $L_\fb \subseteq T^*C^-\oplus T^*(C^+)^*$ is the ``Lagrangian'' graph of $d\fb$, where $\fb$ is viewed as a function of the variables $p^+$ and $q^-$. It is given by the equations
$$
L_\fb :=\left\{ p^-_\gamma=\kappa_\gamma \frac {\del \fb}{\del q^-_\gamma}\,,\, q^+_\gamma=\kappa_\gamma \frac {\del \fb}{\del p^+_\gamma} \right\}.
$$

Similarly, we have
$$
(\frac 1{s!}\fb^{\odot s})\tcev{g_s} = g_s|_{L_\fb}.
$$
\end{rem}

\begin{lem}\label{lem:morphism}
Any $\fb\in \ol{\LL}$ of degree $2N$ gives rise to a morphism of coalgebras $\Phi:\ol{S}(\AA^+[2N]) \to \ol{S}(\AA^-[2N])$ of degree $0$ satisfying $\Phi(\id)=\id$, given on homogeneous generators by the formula
$$
\Phi(w_1 \odot \dots \odot w_r) = \left((e^\fb)\tcev{D}_{w_1 \odot \dots \odot w_r}\right)_{|p^+=0}.
$$
\end{lem}
\begin{proof}[Remarks on the proof]
In this setting, $\Phi$ is a morphism of coalgebras if and only if it can be written as $\Phi=e^\phi$ for a map $\phi:\ol{S}\AA^+ \to S^1\AA^-$. We claim that the component $\phi^r:S^r\AA^+ \to \AA^-$ of $\phi$ of arity $r$ can be explicitly described: On input words $w_1$, \dots, $w_r$ from $\AA^+$ of total word length $\tau$ it is given by
\begin{equation}\label{eq:mor_explicit}
\phi^r(w_1 \odot \dots \odot w_r) = \frac 1{(\tau-r+1)!} \left(\fb^{\odot (\tau-r+1)}\tcev{D}_{w_1 \odot \dots \odot w_r}\right)_{|p^+=0}.
\end{equation}
Indeed, consider a typical summand
$$ 
\frac 1{n!} \left(\fb^{\odot n} \tcev{D}_{w_1\odot \dots \odot w_r}\right)_{|p^+=0}$$
in the definition of $\Phi$. This is itself a sum of various iterated bracket expressions. To each term contributing to such an iterated bracket we can associate a graph with $r$ vertices labelled by the words $w_i$ and $n$ vertices labelled by $\fb$, which describes the pattern in which the iterated brackets connect the pieces. We can think of this graph as being build up gradually, where each new bracket connects two previously unconnected components to each other by a single edge. In particular, all connected components stay trees at all times. As we are taking exactly $\tau$ Poisson brackets in this situation, the final graph has Euler characteristic $r+n-\tau$. In order for the result to land in $\AA^-$, this graph should be a single tree, so we need $n=\tau-r+1$.

From \eqref{eq:mor_explicit}, we can now also conclude that $\phi^r:S^r\AA^+\to S^1\AA^-$ is homogeneous of degree $(1-r)2N$, so it is of degree $0$ after the shift.

Finally, as $\fb\in \ol{\LL}$, we have
$$
\left((e^\fb)\tcev{D}_\id\right)_{|p^+=0} = \left(e^\fb \odot \id \right)_{|p^+=0} = \id.
$$
\end{proof}
\begin{rem}
The relationship between iterated bracket expressions and graphs used in the above discussion is slightly more subtle than it may appear at first. Indeed, suppose $w_1$ and $w_2$ are two words of word length 2 and $w_3$ is of word length 1 in $\AA$. An expression like
$$
\{\fb^1,\{\{\fb^2, \{\{\fb^2, w_1\},w_2\} \}, w_3\}\}
$$
which occurs as a summand in $\phi^3(w_1 \odot w_2 \odot w_3)$ is then made up of two contributions whose corresponding graphs differ. Indeed, in one summand $w_1$ interacts with both $\fb^2$-terms and $w_2$ interacts with the inner one and with $\fb^1$, whereas in the other summand the roles are switched.
\end{rem}
\begin{ex}\label{ex:identity}
For $\AA^+=\AA^-=\AA$, we have an isomorphism of algebras $\LL \cong \PP$. Note however that the actions of elements from $\PP$ via (iterated) Poisson brackets on $\PP$ and $\LL$ will differ. To keep this point clear, we will use variables $p^+$ and $q^-$ when thinking in terms of $\LL$, and variables $p$ and $q$ with the same index set for $\PP$. The identity morphism $\ol S(\AA[2N]) \to \ol S (\AA[2N])$ corresponds to the element
$$
\ib:= \sum_{\gamma \in \Gamma} \frac 1{\kappa_\gamma} q_\gamma^- p_\gamma^+ \in \ol\LL.
$$
\end{ex}

Now suppose we are given two elements $\hb^\pm \in \ol{\PP}^\pm$ or degree $2N-1$ satisfying
$$
\{\hb^\pm,\hb^\pm\} =0.
$$
As explained in the previous subsection, they give rise to codifferentials $D^\pm:\ol{S}\AA^\pm[2N] \to \ol{S}\AA^\pm[2N]$ of degree $-1$.
They also give rise to codifferentials
$$
\vect{D^-}:=\vect{D_{\hb^-}}:S\LL \to S\LL \qquad \text{and} \qquad S\LL \gets S\LL: \tcev{D^+}:=\tcev{D_{\hb^+}},
$$
extending
$$
\vect{\hb^-}:S\LL \to \LL \qquad \text{and} \qquad \LL \gets S\LL : \tcev{\hb^+}.
$$
\begin{lem}\label{lem:comp}
Suppose $w_1,\dots,w_k\in \AA^+$, $\hb^\pm\in \ol{\PP}^\pm$ are as above, and $\fb\in \ol{\LL}$ is of degree $2N$. Then
$$
\left(e^\fb \tcev{D}_{D^+(w_1\odot \dots \odot w_k)}\right)|_{p^+=0} = \left(\left(e^\fb\tcev{D^+}\right)\tcev{D}_{w_1 \odot \dots \odot w_k}\right)|_{p^+=0}
$$
and
$$
D^-\left(e^\fb \tcev{D}_{w_1 \odot \dots \odot w_k}|_{p^+=0}\right) =
\left(\left(\vect{D^-}e^\fb\right) \tcev{D}_{w_1 \odot \dots \odot w_k}\right)|_{p^+=0}. \hfill\qed
$$
\end{lem}
It follows that in order for the coalgebra map $\Phi:S(\AA^+[2N]) \to S(\AA^-[2N])$ associated to an element $\fb\in \ol{\LL}$ of degree $2N$ as above to commute with the codifferentials $D^\pm$ determined by $\hb^\pm\in \ol{\PP}^\pm$, we must have
$$
\vect{D^-}e^\fb = e^\fb\tcev{D^+}.
$$
By an earlier remark, we have
$$
\vect{D^-}e^\fb = \hb^-|_{L_\fb} \odot e^\fb \qquad \text{and} \qquad
e^\fb\tcev{D^+} = e^\fb \odot \hb^+|_{L_\fb}.
$$
So we find
\begin{lem}\label{lem:char_morphism}
Suppose $\hb^\pm \in \ol{\PP}^\pm$ are Poisson self-commuting elements of degree $2N-1$, and $\fb\in \ol\LL$ is an element of degree $2N$. Then the coalgebra map $\Phi:\ol{S}(\AA^+[2N]) \to \ol{S}(\AA^-[2N])$ associated to $\fb$ intertwines the codifferentials $D^\pm$ induced by $\hb^\pm$ if and only if
\begin{equation*}
\hb^-|_{L_\fb}= \hb^+|_{L_\fb}.  \hfill{\qed}
\end{equation*}
\end{lem}
In particular, under the condition formulated in the lemma, $\Phi$ can be interpreted as an $L_\infty$-morphism between the $L_\infty$ algebra structures on $ \AA^\pm[2N+1]$ induced by $D^\pm$.

\subsection{Composition of morphisms}

Suppose we have three structures $(\AA^+,\hb^+)$, $(\AA,\hb)$ and $(\AA^-,\hb^-)$ as above, and consider $\LL^+:=SC \otimes S(C^+[-2N])^*$ and $\LL^-:=SC^- \otimes S(C[-2N])^*$. According to Lemma~\ref{lem:morphism}, elements $\fb^+\in \ol\LL^+$ and $\fb^-\in \ol\LL^-$ define morphisms of coalgebras
$$
\Phi^+:\ol{S}(\AA^+[2N]) \to \ol{S}(\AA[2N]) \qquad \text{and} \qquad 
\Phi^-:\ol{S}(\AA[2N]) \to \ol{S}(\AA^-[2N]).
$$
Our next lemma gives a description of the composition $\Phi:=\Phi^-\circ \Phi^+$. To this end, let $\LL:=SC^- \otimes S(C^+[-2N])^*$, and denote by $L \subseteq T^*C^-\oplus T^*C^*\oplus T^*(C^+)^*$ the submanifold given by the equations
$$
p_\gamma = \kappa_\gamma \frac {\del \fb^+}{\del q_\gamma} \qquad \text{and} \qquad
q_\gamma = \kappa_\gamma \frac {\del \fb^-}{\del p_\gamma},
$$
where $\gamma$ runs over the index set $\Gamma$ associated to $C$.

Given $\fb^+\in \LL^+$ and $\fb^-\in \LL^-$, we now  follow \cite{EGH00} and define $\fb\in \LL$ as
\begin{equation}\label{eq:comp_mor}
\fb(p^+,q^-) := \left(\fb^+(p^+,q) +\fb^-(p,q^-) -\sum_\gamma \frac 1{\kappa_\gamma} q_\gamma p_\gamma \right)_{|L},
\end{equation}
where $L$ is as above.

For what follows, it is useful to understand the mechanics of this expression. We start by considering the first term. Restricting $\fb^+$ to $L$ means replacing each occurence of a $q$-variable by a corresponding derivative of $\fb^-$. As we have argued in the previous subsection in similar circumstances, this has the  same effect as taking $r$ brackets with $\fb^-$. That in turn might introduce ``free'' instances of the variables $p$. Replacing these corresponds to taking new iterated brackets with $\fb^+$ again. This continues until no new $p$s or $qs$ have been introduced. The combinatorics of these iterated replacements can be captured by a tree with vertices labelled by $\fb^+$ and $\fb^-$, and an edge for each instance of a bracket ``connecting'' two such terms. If such a tree has $r$ vertices labelled with $\fb^+$, then it appears $r$ times in $\fb^+(p^+,q)_{|L}$, as it can be built starting from any of these vertices. Similarly, if the tree has $n$ vertices labelled $\fb^-$, then it appears $n$ times in $\fb^-(p,q^-)_{|L}$. Finally, the tree can also be built starting from each of its $\tau$ edges, so it will appear exactly $\tau$ times in the last term. In summary, each such tree $T$ is counted in the expression \eqref{eq:comp_mor} exactly $n+r-\tau=\chi(T)=1$ time.

%
\begin{lem}\label{lem:comp_mor}
Let $\Phi^\pm$ be morphisms as above associated to elements $\fb^\pm \in \ol\LL^\pm$. Then the composition $\Phi:=\Phi^-\circ \Phi^+$ is associated to the element $\fb\in \ol\LL$ given by \eqref{eq:comp_mor}.
\end{lem}

\begin{proof}[Sketch of proof]
From the definitions, we find that
\begin{align*}
\Phi^-(\Phi^+(w_1 \odot \dots \odot w_k)) &= \Phi^-\left(\left((e^{\fb^+})\tcev{D}_{w_1 \odot \dots \odot w_r}\right)_{|p^+=0}\right)\\
&= \left(\left(e^{\fb^-}\right)\tcev{D}_{\left((e^{\fb^+})\tcev{D}_{w_1 \odot \dots \odot w_r}\right)_{|p^+=0}}\right)_{|p=0}.
\end{align*}
With an obvious extension of our previous notations, one checks that this can be rewritten as
$$
\left(\left(\left(e^{\fb^-}\right)\tcev{D}_{e^{\fb^+}}\right)_{|p=0}\tcev{D}_{w_1 \odot \dots \odot w_r}\right)_{|p^+=0}.
$$
Hence it remains to understand the term $\left(\left(e^{\fb^-}\right)\tcev{D}_{e^{\fb^+}}\right)_{|p=0}$. As in the outlined proof of Lemma~\ref{lem:morphism}, a typical term
$$
\left(\frac 1{n!} (\fb^-)^{\odot n}\right)\tcev{D}_{\frac 1{r!} (\fb^+)^{\odot r}}
$$
corresponds to iterated brackets whose combinatorics can be described by a union of trees with $n$ vertices labelled by $\fb^-$ and $r$ vertices labelled by $\fb^+$, and whose $\tau$ edges each correspond to a bracket where two of the terms interact. In order to find $\fb$, we need to find the connected building blocks, i.e. those summands where the underlying graph is a tree. As we have argued just before the statement of the lemma, the expression \eqref{eq:comp_mor} precisel captures the sum of all such trees (with varying $n$, $r$ and $\tau$). 
\end{proof}

\subsection{Maurer-Cartan elements and twisted structures}
\label{ssec:mc}

In this subsection, we assume that the ground ring $\Lambda$ is {\em filtered}, in the sense that for each $\rho\in \R$ we have an ideal $\Lambda_\rho \subset \Lambda$ such that for all $\rho_1,\rho_2\in \R$ we have
\begin{itemize}
\item if $\rho_1<\rho_2$ then $\Lambda_{\rho_1} \supseteq \Lambda_{\rho_2}$, and
\item $\Lambda_{\rho_1} \cdot \Lambda_{\rho_2} \subseteq \Lambda_{\rho_1+\rho_2}$.
\end{itemize}
Assigning all generators filtration degree $0$, we get an induced filtration $\AA_\rho$ on $\AA$, which satisfies $\AA_{\rho_1} \cdot \AA_{\rho_2} \subseteq \AA_{\rho_1+\rho_2}$. We denote by $\AA^\comp$ the completion of $\AA$ with respect to this filtration.

Consider $\AA$ and an element $\hb\in \ol\PP$ of degree $2N-1$ satisfying $\{\hb,\hb\}=0$ as before, and note that the coderivation $D=D_\hb$ naturally extends to $\ol{S}\AA^\comp[2N]$.
For a homogeneous element $a\in \AA^\comp$ of even degree and positive filtration degree, we can consider the element
$$
e^a = 1+ \sum_{k=1}^\infty \frac {a^{\odot k}}{k!} \in S \AA^\comp[2N],
$$
where $S\AA^\comp[2N]$ is also completed with respect to its wordd filtration. One easily checks that for two homogeneous elements $a,b\in \AA^\comp$ of even degrees one has
\begin{equation}\label{eq:exp}
e^{a+b}=e^a\odot e^b.
\end{equation}
\begin{defn}
An element $\ab \in \AA^\comp$ of degree $2N$ and positive filtration degree is called a {\em Maurer-Cartan element} for $(\AA,\hb)$ if
$$
D\left(e^\ab\right)=0.
$$ 
\end{defn}
Given a Maurer-Cartan element  $\ab\in \AA^\comp$, one straightforwardly checks that the map $\Psi^\ab: \ol{S}\AA^\comp[2N] \to \ol{S}\AA^\comp[2N]$, defined on $w=w_1 \odot \dots \odot w_r$ as
\begin{equation}\label{eq:psi^a}
\Psi^\ab(w):=e^\ab \odot w,
\end{equation}
is a map of coalgebras (but note that $\Psi^\ab(\id)=e^\ab \odot \id$). In view of \eqref{eq:exp}, $\Psi^\ab$ is an isomorphism with inverse $\left(\Psi^\ab\right)^{-1}=\Psi^{-\ab}$. It follows that
$$
D^\ab:= \Psi^{-\ab} \circ D \circ \Psi^\ab
$$
is again a coderivation on $\ol{S}\AA^\comp[2N]$ of degree $-1$. 

\begin{defn}
The coderivation $D^\ab$ on $\ol{S}(\AA^\comp[2N])$ as above is called the {\em twisting by $\ab$} of the original coderivation $D$.
\end{defn}

\begin{rem}
Going through the definitions, one can check that the coderivation $D^\ab$ is induced by the element
$$
\hb^\ab :=  \hb + \{\hb,\ab\} + \frac 12 \{\{\hb,\a\},\ab\} + \frac 16 \{\{\{ \hb,\ab\},\ab\},\ab\} + \dots .
$$
The condition that $\ab$ is a Maurer-Cartan element precisely guarantees that $\hb^\ab\in \ol\PP^\comp$, i.e. $\hb^\ab|_{p=0}=0$, or equivalently $D^\ab(\id)=0$.
\end{rem}

\subsection{Maurer-Cartan elements and morphisms}

There are two ways in which Maurer-Cartan elements interact with morphisms.

\subsubsection{Morphisms from $\fb\in \LL$}
\label{ssec:gen_mor}

In the context of considering nonexact symplectic cobordisms in SFT, one comes across element $\fb \in \LL$ which are not in $\ol\LL$. In order to make sense of things in this context, we need to assume that everything is filtered. Moreover, it turns out to be useful to split $\fb$ as
$$
\fb=\fb^0+\fb',
$$
where as before $\fb^0=\fb|_{p^+=0}$ is the part of $\fb$ which is independent of the $p^+$-variables. We assume that $\fb^0$ has positive filtration degree. Going through the argument for defining a morphism $\Phi$ as in Lemma~\ref{lem:morphism}, we find that not much changes except for the fact we no longer have $\Phi(\id)=\id$, but rather
$$
\Phi(\id)=e^{\fb^0}\odot \id\in S\AA^{-,\comp}[2N].
$$
It turns out that the equation
\begin{equation}\label{eq:gen_mor}
\hb^+|_{L_\fb}=\hb^-|_{L_\fb}
\end{equation}
is still equivalent to 
$$
\vect{D^-}e^\fb = e^\fb\tcev{D^+}
$$
in this setting, and considering the part of this equation of degree $0$ in $p^+$ gives
$$
\vect{D^-}e^{\fb^0}=0.
$$
In particular, under the assumption \eqref{eq:gen_mor}, one can view $\fb^0$ as a Maurer-Cartan element in $(\AA^{-,\comp}[2N],D^-)$. The rest of the argument in the previous subsection then goes through to show that the morphism of coalgebras $\Phi':\ol{S}(\AA^{+,\comp}[2N]) \to \ol{S}(\AA^{-,\comp}[2N])$ defined using $\fb':=\fb-\fb^0$ intertwines the codifferentials $D^+$ and $D^{-,\fb^0}$ induced by the elements $\hb^+\in \PP^+$ and $\hb^{-,\fb^0}\in \PP^{-,\comp}$.

\subsubsection{Pushforward of Maurer-Cartan elements}

Let $(\AA^\pm,\hb^\pm)$ be as before, and suppose $\fb\in \ol \LL$ satisfies $\hb^+|_{L_\fb}=\hb^-|_{L_\fb}$, so that the morphism associated to $\fb$ intertwines the codifferentials $D^\pm$ associated to $\hb^\pm$.

\begin{lem}
Let $\Phi:\ol S(\AA^{+,\comp}[2N]) \to \ol S(\AA^{-,\comp}[2N])$ be the morphism associated to $\fb \in \ol\LL$ as above, and let $\ab \in \AA^{+,\comp}$ be a Maurer-Cartan element. Then there exists a unique Maurer-Cartan element $\fb_*(\ab)\in \AA^{-,\comp}[2N])$ such that
$$
\Phi(e^\ab) = e^{\fb_*(\ab)}.
$$
\hfill$\qed$
\end{lem}
In fact, $\fb_*(\ab)$ can be explicitly described as
$$
\fb_*(\ab)= \sum_{r=1}^\infty \frac 1{r!}\phi^r(a^{\odot r}),
$$
where $\phi^r:S^r(\AA^{+,\comp}[2N]) \to S^1(\AA^{-,\comp}[2N])$ are the defining components of $\Phi$.

Consider the self-maps
$$
\Psi^\ab:\ol S(\AA^{+,\comp}[2N]) \to \ol S(\AA^{+,\comp}[2N])
$$
and
$$
\!\!\!\!\!\Psi^{\fb_*(\ab)}:\ol S(\AA^{-,\comp}[2N]) \to \ol S(\AA^{-,\comp}[2N])
$$
associated to the Maurer-Cartan elements $\ab$ and $\fb_*(\ab)$ as in \eqref{eq:psi^a}, and define the ``$\ab$-twisted morphism'' $\Phi^\ab:\ol S(\AA^{+,\comp}[2N]) \to \ol S(\AA^{-,\comp}[2N])$ as
$$
\Phi^\ab:=\Psi^{-\fb_*(\ab)} \circ \Phi \circ \Psi^{\ab}.
$$
It is now straightforward to check from the definitions that the morphism $\Phi^\ab$ intertwines the differentials $D^\ab$ and $D^{\fb_*(\ab)}$ and satisfies $\Phi^\ab(\id)=\id$.

\begin{rem}
The morphism $\Phi^\ab$ is associated to the element $\fb^\ab \in \ol\LL$ which can be described as follows. The element $(e^\fb)\tcev{D}_{e^\ab}\in S\LL$ is itself of the form $e^\gb$ for some $\gb\in \LL$, with leading term $\gb^0=\fb_*(\ab)$. Then, just as above, $\fb^\ab$ simply consists of all the terms in $\gb$ containing at least one $p^+$, i.e.
$$
\fb^\ab= \gb - \fb_*(\ab).
$$
\end{rem}

\subsection{Linearization}
\label{ssec:lin}

In this subsection, we discuss what can be said about $(\AA,\hb)$ assuming that $\hb\in \ol\PP$ satisfies the stronger assumption $\hb\in \wh{\PP}$, i.e. $\hb|_{p=0}=0=\hb_{q=0}$.

Expanding the equation
$$
\{\hb , \hb\}=0
$$
in degrees in $p$ and $q$, we find that
\begin{equation}\label{eq:rel_lin}
\sum_{r_1+r_2=r}\sum_{s_1+s_2=s} \{\hb^{r_1}_{s_1}, \hb^{r_2}_{s_2}\}=0
\end{equation}
for all $r,s \ge 2$.
In particular, the operations
$m^r_s:S^rC \to S^sC$ by
$$
m^r_s(q_{\gamma_1}\odot\dots\odot q_{\gamma_r}):=\{\dots\{\hb^r_s,q_{\gamma_1}\},\dots\},q_{\gamma_r}\}
$$
have degrees $(1-r)2N-1$ and satisfy a series of quadratic relations which reflect \eqref{eq:rel_lin}. In particular, the map $\del=m^1_1:C \to C$ is a boundary operator of degree $-1$, and denoting by $\wh{m^r_s}:\ol{S}C \to \ol{S}C$ the extensions of the above operations, we find
\begin{equation}
\sum_{r_1+r_2=r}\sum_{s_1+s_2=s} \wh{m^{r_1}_{s_1}} \circ_1 \wh{m^{r_2}_{s_2}}=0,
\end{equation}
where for $n\ge r$ the map $\wh{m^{r_1}_{s_1}}\circ_1 \wh{m^{r_2}_{s_2}}:S^nC \to S^{n-r+s}C$ is given by the composition
$$
S^nC \hookrightarrow S^n\AA \stackrel{\vect{\hb^{r_2}_{s_2}}}{\lra} S^{n-r_2+1}\AA \stackrel{\vect{\hb^{r_1}_{s_1}}}{\lra} S^{n-r+1}\AA \to S^{n-r+s}C.
$$
Here the final map in the chain takes a word of words in the $q$s and interprets it as a single word in the $q$s.

Comparing definitions, we see that (up to degree shifts), this is the structure of a dg algebra over the quadratic dual of the Lie bialgebra dioperad in the sense of Gan\cite{Gan03} on $C$ (see also \cite{Mer06} for a discussion). As Gan proves that the Lie bialgebra dioperad is Koszul, this means that $(C,\{m^r_s\}_{r,s\ge 1})$ is a (graded version of a) BiLie$_\infty$-algebra (in the sense of dioperad-algebras). As substructures, we have $\{m^r_1\}_{r\ge 1}$ giving an $L_\infty$-structure on $C[+2N+1]$ and $\{m^1_s\}_{s\ge 1}$ giving a $coL_\infty$-structure on $C[-1]$.

\subsubsection{Linearizing morphisms and Maurer-Cartan elements}

Suppose $\hb^\pm\in \wh\PP^\pm$ are two Poisson self-commuting elements of degree $2N-1$, defining structures $(\AA^\pm,D^\pm)$ which can be linearized. In this situation, an element $\fb\in \wh\LL$ satisfying $\hb^+|_{L_\fb}=\hb^-|_{L_\fb}$ should still give the correct notion of a morphism between the linearized structures.

Here, rather then discussing the general case, we only treat the minimum necessary for the discussion of the work of Siegel below.
\begin{lem}
\label{lem:linearized}
Let $\hb^\pm \in \wh\PP^\pm$ of degree $2N-1$ satisfying $\{\hb^\pm,\hb^\pm\}=0$ and $\fb\in \wh\LL$ of degree $2N$ satisfying  $\hb^+|_{L_\fb}=\hb^-|_{L_\fb}$ be given. Then
\begin{enumerate}[(i)]
\item The parts $\hb^\pm_1$ linear in the $q$-variables give rise to codifferentials
$$
D^{\pm,\lin}:\ol S(C^\pm[2N]) \to \ol S(C^\pm[2N])
$$
of degree $-1$, which describe $L_\infty$-structures on $C^\pm[-2N-1]$.
\item The parts $\fb_1$ linear in the $q^-$-variables give rise to a morphism
of differential coalgebras
$$
\Phi^\lin:(\ol S(C^+[2N]),D^{+,\lin}) \to (\ol S(C^-[2N]),D^{-,\lin}).
$$
\item If $\ab\in \AA^{+,\comp}$ is a Maurer-Cartan element for $(\AA^{+,\comp},\hb^+)$, then its linear part $\ab_1\in C^\comp$ is a Maurer-Cartan element in the linearized $L_\infty$ structure induced from $\hb^+_1$ on $C^{+,\comp}$.
\end{enumerate}
\hfill$\qed$
\end{lem}

\subsection{Augmentations and Linearization}
\label{ssec:aug_lin}

As linearized structures are defined on a considerably smaller complex ($C$ as opposed to $SC$) and allow to access the full $biLie_\infty$ structure, situations where they can be achieved are of special interest. One particularly useful way to systematically construct them is via augmentations, which are morphisms from $(\AA^+,\hb^+)$ to the trivial structure $(\AA^-,\hb^-)=(\Lambda,0)$. To distinguish it from other spaces of mixed power series appearing in this subsection, we denote the one for this special case by $\LL_0$. It consists of pure power series in the variables $p^+$ (i.e. there are no $q$-variables), and so it can be viewed as a subalgebra of any other $\LL$ with the same $\AA^+$. An augmentation then determines and is determined by an element $\fb\in \ol\LL_0$ satisfying $\hb^+|_{L_\fb}=0$.

We will now consider the special case $(\AA^-,\hb^-)=(\AA^+,\hb^+)$, which was already mentioned in Example~\ref{ex:identity}. Fix $\fb\in \ol\LL_0\subseteq \ol\LL$ satisfying $\hb^+|_{L_\fb}=0$. We use it to construct the morphism
$$
\Psi_\fb:\ol S(\AA^+[2N]) \to \ol S(\AA^-[2N])\cong \ol S(\AA^+[2N])
$$
associated via Lemma~\ref{lem:morphism} to the element $\ib+\fb\in \ol\LL$, where $\ib=\sum_{\gamma\in \Gamma^+} \frac 1{\kappa_\gamma}q_\gamma^- p_\gamma^+$ is as in Example~\ref{ex:identity}. By construction, it intertwines the codifferential $D$ on the domain with the codifferential $\Psi_\fb \circ D \circ \Psi_\fb^{-1}$ on the target (here we drop the superscripts $\pm$ to lighten the notation a bit).

\begin{lem}
The codifferential
$$
\Psi_\fb \circ D \circ \Psi_\fb^{-1}:\ol S(\AA[2N]) \to \ol S(\AA[2N])
$$
is associated to an element $\hb_\fb\in \wh\PP$, and so it can be linearized.
\end{lem}

\begin{proof}
We need to show that the projection
$$
\pi_\Lambda \circ \Psi_\fb \circ D \circ \Psi_\fb^{-1}: \ol S(\AA[2N]) \to \Lambda[2N]
$$
onto $\Lambda[2N] \subset \AA[2N] \cong S^1\AA[2N]$ vanishes. 
We have already seen the discussion leading up to Lemma~\ref{lem:char_morphism} that
$$
e^{\ib+\fb}\,\tcev{D^+}=e^{\ib+\fb} \odot \hb|_{L_{\ib+\fb}}.
$$
In particular, the condition $\pi_\Lambda \circ \Psi_\fb \circ D=0$ is equivalent to the condition $\left(\hb|_{L_{\ib+\fb}}\right)|_{q^-=0}=0$. 
But since every term in $\ib$ contains a $q$-variable, we have
$$
\left(\hb|_{L_{\ib+\fb}}\right)|_{q=0}=\hb|_{L_\fb},
$$
and the later is zero by the assumption that $\fb\in \ol\LL_0$ defines a morphism from $(\AA,\hb)$ to $(\Lambda,0)$.
\end{proof}
\bigskip

\begin{rem}
The element $\hb_\fb\in \wh\PP$ in the statement of this lemma admits the following explicit description: Consider the action of $S\LL_0$ on $\PP$ by setting
$$
\vect{c_1 \odot \dots \odot c_r}(g):=\{c_1,\{\dots,\{c_r,g\}\dots\}.
$$
Note that, because the $c_i$ Poisson commute, this is well-defined by an argument analogous to the proofs of Lemmata~\ref{lem:codiff} and \ref{lem:symmetric}.

A straightforward unwinding of the definitions now yields that
$$
\hb_\fb = \vect{e^\fb}\,\hb.
$$
Alternatively, it can be described as $\hb|_{L_{\ib+\fb}}$.
Both expressions correspond to ``composing'' {\em some} of the outputs of $\hb$ with inputs of $\fb$.
\end{rem}

\subsubsection*{Augmenting Maurer-Cartan elements}

Now suppose that $\ab \in \AA^\comp$ is a Maurer-Cartan element for the structure associated to some $\hb\in \ol\PP$. For any $\fb\in \ol\LL_0$ satisfying $\hb|_{L_\fb}=0$, we can consider the morphism
$$
\Psi_\fb: \ol S(\AA^\comp[2N]) \to \ol S(\AA^\comp[2N])
$$
associated to the element $\ib+\fb \in \ol\LL$ as above.
We then get the pushforward Maurer-Cartan element $(\ib+\fb)_*\ab\in \AA^\comp$ for the augmented (and hence linearizable) structure associated to $\hb_\fb\in \wh\PP$. The part $\left((\ib+\fb)_*\ab\right)_1$ of this Maurer-Cartan element which is linear in the $q$-variables will define a Maurer-Cartan element for the linearized structure, i.e. the $L_\infty$-structure on $C[2N+1]$ defined using the $\{(\hb_\fb)^r_1\}_{r\ge 1}$ (as discussed in \S~\ref{ssec:lin} above).





\section{Relation to the SFT formalism from \cite{EGH00}}
\label{sec:rsft}

Here we very briefly review the relevant formalism from the foundational paper by Eliashberg, Givental and Hofer~\cite{EGH00}. We do assume that the reader is familiar with most of the content, and just highlight the relevant points for our discussion of {\em rational SFT} here.

Given a contact manifold $(V^{2n-1},\alpha)$ with all periodic orbits
non-degenerate, the formalism uses the structures introduced in section~\ref{sec:algebra} with the indexing set $\Gamma$ consisting of all (good) Reeb orbits $\gamma$. Here the $\Z$-grading exists under favourable circumstances, and comes from the Conley-Zehnder indices of the Reeb orbits via
$$
|q_\gamma|=n-3+CZ(\gamma). 
$$
The integer $N$ introduced above equals $n-3$ in this context, so that
$$
|p_\gamma|=n-3-CZ(\gamma),
$$
and $\kappa_\gamma$ is the geometric multiplicity of the Reeb orbit.

There are several options for the coefficient ring:
\begin{itemize}
\item In the simplest situations, one can choose $\Lambda=\Q$. This works whenever one is interested in invariants of single contact manifolds, or when all considered cobordisms are exact, and when no additional constraints beyond the asymptotics at the punctures are implemented.
\item If one wants to include the discussion of non-exact cobordisms, one can instead use $\Lambda=\Lambda_{\mathrm{Nov}}$, the universal Novikov ring
$$
\Lambda_{\mathrm{Nov}}:= \left\{ \left.\sum_i c_i\lambda^{a_i} \,\right|\, c_i\in \Q, \, a_i \ge 0, \, \lim_{i\to \infty}a_i=\infty\right\}.
$$
It is filtered by the minimal exponent of a term with non-zero coefficient.
\item In the most general setting, one can use for $\Lambda$ a graded polynomial ring in (potentially infinitely many) variables $t_i$, which correspond to possible point constraints (and the degree of $t_i$ corresponds to the codimension of the constraint) over the universal Novikov ring $\Lambda_{\mathrm{Nov}}$. This is not exactly what is done in \cite{EGH00}, where they use the Novikov completion of $\Z[H_2(V;\Z)]$, but our version just amounts to a convenient repackaging of the information, and allows us to treat all contact manifolds and cobordisms over the same coefficients (for the constraints, ``same'' has to be suitably interpreted).
\end{itemize}

\begin{rem}
We use {\em point constraint} as a generic term for any constraint on a map which is imposed at a single point of the domain. This encompasses all of the different constraints Siegel~\cite{Sie19} and Moreno and Zhou~\cite{MZ20} use, if one interprets Siegel's higher order self-contact constraints in terms of maps to blow ups of the target symplectic manifold as in \cite[\S~5.3]{Sie19}.
\end{rem}
To a contact manifold as above, {\em rational SFT} associates an element
$\hb \in \ol\PP$ of degree $2N-1$, where the coefficient of any monomial in the $q_\gamma$, $p_\gamma$ and $t_i$ is the count of rigid rational holomorphic curves of finite Hofer energy in $\R \times V$ with respect to a cylindrical almost complex structure $J$ and with positive asymptotics corresponding to the $p$-variables, negative asymptotics corresponding to the $q$-variables, and subject to constraints at (distinct) interior marked points corresponding to each of the $t_i$. This count will typically depend on the auxillary choices, such as the almost complex structure and the choice of virtual fundamental chain technique needed to deal with failure of regularity. 
The compactness, transversality and gluing results give a description of the boundaries of moduli spaces rational holomorphic curves of expected dimension 1. The algebraic consequence of this description for the generating function $\hb$ has a very concise form:
\begin{thm}\cite[Lemma~2.2.6]{EGH00}
The rational Hamiltonian $\hb\in \ol\PP$ satisfies
\begin{equation}\label{eq:rsft1}
\{\hb,\hb\} =0.
\end{equation}
\end{thm}
\begin{rem}
For later use, we point out that equation \eqref{eq:rsft1} can be expanded in the constraints, so that each monomial in the $t_i$ yields a separate equation. Temporarily denote by $\hb'$ the part of $\hb$ independent of $t$ (i.e. without point constraints), and by $\hb(t_1)$ the part of $\hb$ {\em linear} in a single constraint $t_1$. Then we have
$$
\{\hb',\hb'\} = 0 \quad \text{and} \quad
2\{\hb',\hb(t_1)\}= \{\hb(t_1),\hb'\} + \{\hb',\hb(t_1)\} =0.
$$
This will be used in \S~\ref{sec:moreno_zhou} below.
\end{rem}

Next suppose $W$ is a compact symplectic cobordism with positive end $(V^+,\alpha^+)$ and negative end $(V^-,\alpha^-)$. Associated to the two ends we have $\hb^\pm\in \ol{\PP^\pm}$. The counts of rigid rational holomorphic curves associated to choices of almost-complex structure and perturbation data suitably consistent with the choices at the ends can be packaged into an element $\fb \in \LL$ of degree $2N$. In this situation, the analysis of compactifications of moduli spaces of rational holomorphic curves in the completion of $W$ of expected dimension 1 yields
\begin{thm}\cite[Theorem~2.3.6]{EGH00}
The rational potential $\fb\in \LL^\comp$ of a cobordism $W$ and the rational Hamiltonians $\hb^\pm$ of the symplectizations of the two ends are related by
\begin{equation}\label{eq:rsft2}
\hb^-|_{L_\fb} = \hb^+|_{L_\fb}.
\end{equation}
\end{thm}

\begin{rem}
\begin{itemize}
\item When the cobordism $W$ is exact, one does not need to complete and we have $\fb\in \ol\LL$ (since for energy reasons every holomorphic curve in $W$ must have at least one positive end). In particular, $\fb|_{t=0}$ (the part of the potential without constraints) will give rise to a morphism
$$
\Phi:(\ol S(\AA^+[2N]), D^+) \to (\ol S(\AA^-[2N],D^-)
$$
as in \S~\ref{ssec:mor} above.
\item When the cobordism $W$ is not exact, we generally expect $\fb\notin \ol\LL$, and so one proceeds as in \S~\ref{ssec:gen_mor} by splitting $\fb|_{t=0}=\fb^0+\fb'$ and interpreting $\ab=\fb^0\in \AA^{-,\comp}$ as a Maurer-Cartan element for the structure at the negative end, and using $\fb'$ to construct the corresponding morphism
$$
\Phi:(\ol S(\AA^{+,\comp}[2N]), D^+) \to (\ol S(\AA^{-,\comp}[2N],D^{-,\ab}).
$$

\end{itemize}
\end{rem}

We stress again that \eqref{eq:rsft2} holds with constraints present, provided that the constraints in $V^\pm$ and $W$ labelled by the same $t$-variable are suitably compatible (cf.~\cite[\S~2.3.1]{EGH00} for some more details on this). Most relevant to us are two particular situations:
\begin{itemize}
\item In \S~\ref{sec:siegel}, we will use one or several point constraints in the interior of a compact cobordism $W$ with empty negative end. Algebraically, this has the effect that the left hand side of \eqref{eq:rsft2} vanishes, and that the $t$-variables corresponding to the constraints only appear in $\fb$, but not in $\hb^+$. Expanding \eqref{eq:rsft2} in the constraints yields one equation of the form
\begin{equation}\label{eq:mor_with_constraints}
\left(e^{\fb}(T)\right)\tcev{D^+}=0
\end{equation}
for each monomial $T$ in the constraints $t_i$, where $e^\fb(T)$ denotes the part of $e^\fb$ corresponding to the fixed monomial $T$. Temporarily denote by $\fb'=\fb|_{t=0}$ the part of $\fb$ corresponding to curves without constraints, and by $\fb(T)$ the part of $\fb$ corresponding to a given monomial $T$. With these conventions, we have for example
\begin{align*}
e^\fb(t_1)&=\fb(t_1) \odot e^{\fb'},\\
e^\fb(t_1t_2)&=\left(\fb(t_1t_2)+\fb(t_1)\odot \fb(t_2)\right)\odot e^{\fb'} \qquad \text{and}\\
e^\fb(t_1^2)&=\left(\fb(t_1^2)+\frac 12 \fb(t_1)\odot \fb(t_1)\right) \odot e^{\fb'},
\end{align*}
from which the general pattern should be clear. For each $T$, $e^\fb(T)$ is the product of $e^{\fb'}$ with a suitable polynomial in the components of $\fb$ with constraints which reflects the possible partitions of $T$ and potential symmetries of the resulting configurations.
\item In \S~\ref{sec:moreno_zhou} below we will use the case that we have a single $t$-variable corresponding to a point constraint on $\R\times \{x^\pm\} \subseteq \R \times V^\pm$ on the ends, joined by a path of constraints connecting $x^+$ to $x^-$ in the compact part of the cobordism. With the conventions of \cite{EGH00}, this would be modelled by a Poincare dual $(2n-1)$-form, and the degree of the corresponding variable $t_1$ would be $2n-3=2N+3$. The part of \eqref{eq:rsft2} which is linear in this variable $t_1$ then has the form
\begin{equation}
\hb^-(t_1)|_{L_{\fb'}} - \hb^+(t_1)|_{L_{\fb'}} = \{\fb(t_1),(\hb^+)'\}|_{L_{\fb'}} - \{(\hb^-)',\fb(t_1)\}|_{L_{\fb'}},
\end{equation}
where we have used the same conventions as above, i.e. primed expressions correspond to $T=0$. Up to a change of notation, this easily translates to equation \eqref{eq:g-homotopy} appearing in \S~\ref{sec:moreno_zhou} below.
\end{itemize}

\section{The morphisms used by Siegel}
\label{sec:siegel}

With the language developed above, it is now fairly straightforward to describe the algebraic nature of the morphisms that Siegel uses to define his capacities.

Siegel works in the setting of a contact manifold $(V^+,\alpha^+)$ with a symplectic filling $(W,\omega)$. As already mentioned in our summary of rational SFT, he considers holomorphic curves with (various types of) constraints at interior points of $W$, so formally we are in the setting of \eqref{eq:mor_with_constraints}. One possible interpretation of this equation is that for each monomial $T$ in the constraints we have a map
$$
\Phi(T): \ol S(\AA^+[2N]) \to \ol S \Lambda, \qquad\Phi(T)(w_1\odot \dots \odot w_r)=\left(\left(e^\fb(T)\right)\tcev{D}_{w_1\odot \dots \odot w_r}\right)|_{p=0}
$$
whose degree depends on the degree of $T$, but which vanishes on the image of $D^+$. Note that $\Phi(T)$ is {\em not} a morphism of coalgebras unless $T=0$, so for $T\ne 0$ these maps are {\em not} augmentations. Nonetheless, as they vanish on $\im D^+ \subseteq \ol S(\AA^+[2N])$, they do descend to well-defined maps on $H_*(\ol S(\AA^+[2N],D^+)$.

In the restricted setting of studying a contact manifold $(V^+,\alpha^+)$ as the boundary of a given compact symplectic domain $(W,\om)$, the contact form $\alpha^+$ will be well-defined up to addition of an exact 1-form $df$. This has the effect of giving the whole theory a well-defined filtration by ``action'', which is essential in Siegel's work. Many of his higher capacities are obtained by looking at specific monomials $T$ in the constraints and specific elements $b\in S\Lambda$, and looking for the infimum of exponents $a$ such that $\lambda^a \cdot b$ is in the image of $\Phi(T)$.

\begin{rem}
The well-definedness of these capacity requires enough of the machinery needed to prove invariance of (rational) SFT to be able to conclude that these numbers are indeed independent of any auxillary choices. 
\end{rem}

An important property of a symplectic capacity is its monotonicity with respect to symplectic embeddings, meaning that whenever there is a symplectic embedding $(W^-,\omega^-) \subseteq (W^+,\omega^+)$, the inequality $c(W^-,\omega^-) \le c(W^+,\omega^+)$ should hold. In the situation of interest to us, $(W^\pm,\omega^\pm)$ are compact Liouville domains with convex boundaries $V^\pm$, meaning that the Liouville forms $\lambda^\pm$ on $W^\pm$ restrict to (positive) contact forms $\alpha^\pm$ on $V^\pm$. By shrinking $W^-$ slightly if necessary, we can arrange that $W^- \subseteq \operatorname{int}(W^+)$, so that we get a compact symplectic cobordism $W=W^+ \setminus W^-$ with positive end $(V^+,\alpha^+)$ and negative end $(V^-,\alpha^-)$. We assume that the constraints that are used to define our capacity $c$ are localized at points in $W^-\subseteq W^+$. By a neck stretching argument, we may decompose the relevant map $\Phi^+(T): \ol S(\AA^+[2N]) \to \ol S\Lambda[2N]$ as the precomposition of $\Phi^-(T): \ol S(\AA^-[2N]) \to \ol S\Lambda[2N]$ with the cobordism map.

Now there are two cases to consider, depending on whether or not the cobordism $W$ is exact (equivalently, whether or not the embedding $W^- \hookrightarrow W^+$ is exact). In the exact case, the potential $\fb$ of the cobordism $W$ is an element of $\ol\LL$, and so by Lemma~\ref{lem:morphism} it straightforwardly translates into a morphism  $\Phi: \ol S(\AA^+[2N]) \to \ol S(\AA^-[2N])$. We will then have $\Phi^+(T)=\Phi^-(T)\circ \Phi$, and so if $w \in \ol S(\AA^+[2N])$ gets mapped to $\lambda^a \cdot b \in \ol S\Lambda[2N]$ by $\Phi^+(T)$, then $\Phi(w)$ is mapped to the same element under $\Phi^-(T)$. This clearly proves monotonicity in this case.

When the cobordism $W$ is not exact, we split the potential $\fb$ into
$\fb=\fb^0+\fb'$, where $\fb^0$ is the part which is independent of the $p^+$-variables, and $\fb' \in \ol\LL$. As described in \S~\ref{ssec:gen_mor}, $\fb^0\in \AA^{-,\comp}$ will be a Maurer-Cartan element, and $\fb'$ will define a morphism $\Phi':\ol S(\AA^{+,\comp}[2N] \to \ol S(\AA^{-,\comp})$ which intertwines the codifferential $D^+$ induced from $\hb^+$ with the codifferential $D^{-,\fb^0}$ induced from the twisted Hamiltonian $\hb^{-,\fb^0}\in \ol\PP^{-,\comp}$. We then have
$$
\Phi^+(T) = \Phi^-(T) \circ \Psi^{\fb^0} \circ \Phi',
$$
where $\Psi^{\fb^0}:\ol S(\AA^{-,\comp}[2N] \to \ol S(\AA^{-,\comp}[2N])$ is the multiplication with $e^{\fb^0}$ already considered in \S~\ref{ssec:mc}. In particular, if $w\in \ol S(\AA^+[2N])$ gets mapped to $\lambda^a \cdot b \in \ol S\Lambda[2N]$ by $\Phi^+(T)$, then $e^{\fb^0} \odot \Phi'(w)$ is mapped to the same element under $\Phi^-(T)$.  As $e^{\fb^0}$ only makes sense in the completed setting, this proves monotonicity for general symplectic embeddings under the assumption that one uses completed complexes in the definition of the capacities.

In the special case that one considers a single point constraint $T=t_1$, there is also a linearized version, which looks at the corresponding maps
$$
\phi(t_1): \ol S C[2N] \to S^1\Lambda=\Lambda
$$
on the linearized theory, whose defining components
$$
\phi^r(t_1):S^rC[2N] \to \Lambda
$$
for $r\ge 1$ is given (in the notation of the previous section) by the part $\fb^r(t_1)$ of $\fb(t_1)$ which is of degree $r$ in $p^+$. The invariance argument then requires the linearized structures mentioned in Lemma~\ref{lem:linearized}.

\section{The invariants used by Moreno and Zhou}
\label{sec:moreno_zhou}

In this section, we comment on the purely algebraic aspects of the invariants of Moreno and Zhou. Specifically, we explain how the torsion and the order can be described in our version of the formalism. Throughout this section, we specialize to the ground ring $\Lambda=\Q$.

Let $\AA$ and $\PP$ be as before, and let $\hb\in \ol\PP$ be an element of degree $2N-1$ with
$$
\{\hb,\hb\}=0.
$$
As observed at the end of \S~\ref{ssec:structures}, the associated codifferential $D:\ol{S}(\AA[2N]) \to \ol{S}(\AA[2N])$ leaves invariant the submodules $S^{\le k}(\AA[2N]):= \oplus_{r=1}^k S^r(\AA[2N])$. This leads to the following definition (cf. \cite[Def.~2.12]{MZ20}).

\begin{defn}\label{def:torsion}
Let $\AA$ and $\hb\in \ol\PP$ be as above and let $D=D_\hb:\ol{S}(\AA[2N]) \to \ol{S}(\AA[2N])$ be the corresponding codifferential of degree $-1$. The {\em torsion} $T(\AA,\hb)\in \N_0 \cup \{\infty\}$ of $(\AA,\hb)$ is defined as
$$
T(\AA,\hb) := \min \left\{k-1 \,:\, \id\in \im\left(D|_{S^{\le k}(\AA[2N])}\right)\right\},
$$
where as before $\id\in \Q \subseteq \AA[2N] \cong S^1(\AA[2N])$ denotes the unit and we use the usual convention that the minimum of the empty set is $\infty$.
\end{defn}
\begin{rems}\
\begin{itemize}
\item Note that $D(w_1 \odot \dots \odot w_k)$ can only have a summand in $\Q \subseteq \AA[2N] \cong S^1\AA[2N]$ if we have $w_i=q_{\gamma_i}$ for all $i$ and suitable (not necessarily distinct) labels $\gamma_i\in \Gamma$.
\item $T(\AA,\hb)=0$ means that $1\in \im(D^1)$. As $D^1:\AA \to \AA$ is a derivation, this is easily seen to be equivalent to $H_*(\AA,D^1)=0$.
\end{itemize}
\end{rems}

Suppose $(\AA^\pm,\hb^\pm)$ are as in Definition~\ref{def:torsion}, so that $T(\AA^\pm,\hb^\pm)\in \N_0\cup\{\infty\}$ are defined. If $\fb\in \ol\LL$ of degree $2N$ satisfies
$$
\hb^+|_{L_\fb} = \hb^-|_{L_\fb},
$$
so that the morphism $\Phi$ from Lemma~\ref{lem:morphism} is defined, then (as originally observed in \cite{MZ20})
$$
T(\AA^+,\hb^+) \ge T(\AA^-,\hb^-).
$$
Indeed, by construction the morphism $\Phi$ maps $\ol{S}^{\le k}(\AA^+[2N])$ into $\ol{S}^{\le k}(\AA^-[2N])$, and so any element $a\in \ol{S}^{\le k}(\AA^+[2N])$ satisfying $D^+(a)=\id$ will give rise to an element $\Phi(a) \in \ol{S}^{\le k}(\AA^-[2N])$ satisfying
$$
D^-(\Phi(a))=\Phi(D^+(a))=\Phi(\id)=\id.
$$

\begin{rem}
Note that the trivial structure with $\AA=\Q$ and $\hb=0$ clearly has $T(\Q,0)=\infty$, so finiteness of torsion obstructs the existence of an augmentation.
\end{rem}

\begin{rem}
One can also consider an alternative invariant defined as
$$
\wt{T}(\AA,\hb):= \min \left\{ k-1 \,:\, \id \in \im \left(\pi \circ D|_{S^{\le k}(\AA[2N])}\right)\,\right\},
$$
where $\pi$ denotes the projection of $\ol{S}(\AA[2N])$ to $\Q \subseteq S^1(\AA[2N]) \cong \AA[2N]$. Clearly one has
$$
\wt{T}(\AA,\hb) \le T(\AA,\hb)
$$
for all $(\AA,\hb)$ as above. While it might be easier to compute, the overall usefulness of this invariant is somewhat unclear, as there seems to be no obvious analogue of monotonicity under morphisms.
\end{rem}

Now, in addition to the above, assume the stricter condition that $\hb\in \wh{\PP}$. This automatically gives $T(\AA,\hb)=\infty$, and we have seen that one way to achieve it is to twist with respect to an augmentation.
Suppose further that we are given another homogeneous element $\gb\in \ol\PP$ of degree $d$ such that
\begin{equation}\label{eq:g}
\{\hb,\gb\}=0.
\end{equation}
Associated to $\hb$ and $\gb$ we have the coderivations $D_\hb$ and $D_\gb$ defined via iterated brackets as in \S~\ref{ssec:structures}. 
The new condition $\hb|_{q=0}=0$ translates into $\pi \circ D_\hb=0$, where $\pi$ is the same projection to $\Lambda[2N]\cong \Q$ as above, so that by Lemma~\ref{lem:codiff} 
$$
\pi \circ D_\gb \circ D_\hb = (-1)^{|\gb|} \pi \circ D_\hb \circ D_\gb = 0.
$$
It follows that $\pi \circ D_\gb$ descends to a morphism
$$
\pi \circ D_\gb: H_*(S^{\le k}(\AA[2N]),D_\hb) \to \Lambda
$$
of degree $d-4N$ for every $k \ge 1$. Again following Moreno and Zhou (cf.~\cite[Def.~2.17]{MZ20}), we make
\begin{defn}\label{def:order}
In this situation, we define the order of $(\AA, \hb, \gb)$ as
$$
O(\AA,\hb,\gb):= \min\,\{k \,:\, \id \in \im \pi\circ D_\gb|_{H_*(S^{\le k}(\AA[2N]),D_\hb)}\,\},
$$
where again the minimum over the empty set is set to be $\infty$.
\end{defn}

Now suppose we are given $\hb^\pm\in \wh{\PP}^\pm$ of degree $2N-1$ and $\gb^\pm\in \ol\PP^\pm$ of the same degree $d$ satisfying
$$
\{\hb^\pm,\hb^\pm\}=0 \qquad \text{and} \qquad \{\hb^\pm,\gb^\pm\}=0,
$$ 
so that $D^\pm=D_{\hb^\pm}$ and $T(\AA^\pm,\hb^\pm,\gb^\pm)$ are defined. 
\begin{lem}
In this situation, suppose that $\fb\in \wh \LL$ of degree $2N$ satisfies
$$
\hb^-|_{L_\fb}=\hb^+|_{L_\fb},
$$
and suppose there exists an element $\gb\in \LL$ of degree $d$ satisfying
\begin{equation}\label{eq:g-homotopy}
\vect{D_{\gb^-}}e^\fb - e^\fb\tcev{D_{\gb^+}} = \left(e^\fb \odot \gb\right)\tcev{D^+} - \vect{D^-}\left(\gb\odot e^\fb\right).
\end{equation}
Then
$$
O(\AA^+,\hb^+,\gb^+) \ge O(\AA^-,\hb^-,\gb^-).
$$
\end{lem}

\begin{proof}
By assumption $\fb$ defines a morphism $\Phi:\ol{S}\AA^+[2N] \to \ol{S}\AA^-[2N]$ as in Lemma~\ref{lem:morphism}, which by Lemma~\ref{lem:char_morphism} intertwines the codifferentials $D^\pm$ induced from $\hb^\pm$.
%

Let $a\in \ol{S}^{\le k}(\AA^+[2N])$ be an element such that $D^+(a)=0$ (so it defines a homology class in $H_*(\ol{S}^{\le k}(\AA^+[2N],D^+)$) and $\pi \circ D_{\gb^+}(a)=\id$. Then
\begin{align*}
\pi \circ D_{\gb^-}(\Phi(a))
&= \pi \circ \vect{D_{\gb^-}} \left(e^\fb\tcev{D_a}\right)|_{p^+=0}\\
&=\pi \circ \left( \left( \vect{D_{\gb^-}} e^\fb\right)\tcev{D_a}\right)|_{p^+=0}\\
&=\pi \circ \left( \left( e^\fb\tcev{D_{\gb^+}} + \left(e^\fb \odot \gb\right)\tcev{D^+} - \vect{D^-}\left(\gb\odot e^\fb\right)\right)\tcev{D_a}\right)|_{p^+=0} \\
&=\pi \circ \left( e^\fb\tcev{D}_{\vect{D_{\gb^+}}(a)} + \left(e^\fb \odot \gb\right)\tcev{D_{D^+(a)}} - \vect{D^-}\left(\left(\gb\odot e^\fb\right)\tcev{D_a}\right)\right)|_{p^+=0} \\
&=\pi \circ \Phi (D_{\gb^+}(a)) \\
&=\Phi \circ \pi (D_{\gb^+}(a)) \\
&=\Phi(\id) = \id
\end{align*}
Here the first equality just uses the definitions, the second one uses (a slight generalization of) Lemma~\ref{lem:comp}, the third one uses \eqref{eq:g-homotopy}, the forth one again uses (generalizations of) Lemma~\ref{lem:comp}, the fifth one uses $D^+(a)=0$ and $\pi\circ D^-=0$, and the sixth one uses that $\fb\in \wh\LL$, so that $\pi\circ \Phi =  \Phi \circ \pi$, because under this assumption $\Phi$ maps any nontrivial (product of) monomial(s) in $q^+$ into a linear combination of nontrivial (products of) monomials in $q^-$.

As $\Phi$ maps $\ol{S}^{\le k}(\AA^+[2N])$ into $\ol{S}^{\le k}(\AA^-[2N])$, the claim follows.
\end{proof}

In the applications of torsion in rational SFT by Moreno and Zhou~\cite{MZ20}, $\hb$ is the rational Hamiltonian without point constraints for a contact manifold $V$, and monotonicity holds with respect to exact cobordisms.

The order is discussed in the more restricted setting that the structure $(\AA,\hb)$ admits an augmentation $\eps$, and then one considers the twisted versions $\hb_\eps$ and $\gb_\eps$ (in the sense of \S~\ref{ssec:aug_lin}), where $\gb$ is the part of the Hamiltonian with exactly one point constraint on a generic $\R \times \{x\} \subseteq \R \times V$. For monotonicity (with respect to exact cobordisms), one then needs that the augmentations $\eps^\pm$ of $(\AA^\pm,\hb^\pm)$ are compatible, in the sense that $\eps^+=\eps^- \circ \Phi$.

\begin{rem}
As \eqref{eq:g} was the main condition needed to define the order invariant, any cycle defining a homology class in the rational SFT algebra $H_*(\PP,\mathbf D)$ mentioned at the end of \S~\ref{ssec:structures} can be used as an input. Such an order invariant will be monotone with respect to (exact) cobordisms if the elements $g^\pm$ defined on the ends of a cobordism can be related as in \eqref{eq:g-homotopy}. All the single point constraints considered by Siegel (but now with the point in the symplectization $\R \times V$ rather than in some filling) are of this type.

Alternatively, in the presence of a filling, the {\em minimal word length} rather than the action associated to the maps $\Phi(T)$ considered by Siegel will also lead to numerical invariants which are monotone with respect to exact embeddings.
\end{rem}

%
%
%

\end{document}